\documentclass[11pt,]{article}
\usepackage{graphicx}
\usepackage{amssymb}
\usepackage{epstopdf}
\usepackage{amsmath,amsthm,amscd,amssymb}
\usepackage{latexsym}
\usepackage[colorlinks,citecolor=red,pagebackref,hypertexnames=false]{hyperref}
\usepackage{geometry}   
\usepackage{cite}            
\numberwithin{equation}{section}

\theoremstyle{plain}
\newtheorem{theorem}{Theorem}[section]
\newtheorem{lemma}[theorem]{Lemma}

\theoremstyle{definition}

\newtheorem{case[theorem]}{Case}

\theoremstyle{remark}
\newtheorem{remark}[theorem]{Remark}

\numberwithin{equation}{section}

\begin{document}

\title{Pinned Distances in Modules over Finite Valuation Rings}

\author{Esen Aksoy Yazici}

\maketitle


\begin{abstract}
  Let $R$ be a finite valuation ring of order $q^r$ where $q$ is odd and $A$ be a subset of $R$. In the present paper, we prove that there exists a point $u$ in the Cartesian product set $A\times A\subset R^2$ such that the size of the  pinned distance set at $u$ satisfies
$$|\Delta_u(A\times A)|\gg \min\left\{q^r, \frac{|A|^3}{q^{2r-1}}\right\}.$$ 
This implies that if $|A|\ge q^{r-\frac{1}{3}}$, then the set $A\times A$ determines a positive proportion of all possible distances.

\end{abstract}




\section{Introduction} 
\vskip.125in
Erd\H{o}s-Falconer type problems in discrete geometry ask for a threshold on the size of a set so that the set determines the given geometric configurations. These problems have been studied by many authors both in continuous and discrete setting.

In \cite{E46}, Erd\H{o}s observed that  $\sqrt{n}\times \sqrt{n}$ integer grid determines $C( \frac{n}{\sqrt{ log n}})$ distances and he conjectured that the minimum number of distances determined by a  $n$-point set in the plane is indeed $C (\frac{n}{\sqrt{ log n}})$, where $C$ is an absolute constant. Despite many works and progress, the Erd\H{o}s distance problem was open until recently. In 2010, Guth and Katz \cite{GK} employed a polynomial partitioning technique based on  Elekes-Sharir framework to prove that  $n$ points in the plane determine at least $C\frac{n}{log n}$ distances.  This result solved the Erd\H{o}s distance problem  up to a  $\sqrt{\log}$ factor.

The distance problem in finite field plane was first studied by  Bourgain, Katz and Tao in \cite{BKT}. The result in \cite{BKT} was later generalized  by various authors to higher dimensional vector spaces. It was also extended to many other geometric configurations in finite field geometry, see for instance \cite{BHIPR13, BIP, BI, CEHIK12, CHISU, HIKR11, IR07, IRZ12, P} and references therein. In particular, in \cite{P}, Petridis proved the following pinned distance result for Cartesian product subsets of vector spaces over prime fields. 

\vskip.125in
\begin{theorem}\emph{\cite[Theorem 1.1]{P}} Let $p$ be an odd prime and  $A\subset \mathbb{F}_p$. There exist $a,b\in A$ such that
$$|\Delta_{(a,b)}(A\times A)|=\Omega(\min\{p, |A|^{3/2}\}).$$
\end{theorem}

Similar geometric problems in modules over finite cyclic rings  were studied by Covert, Iosevich and Pakianathan in \cite {CIP}. Using a Fourier analytic approach, the authors of  \cite {CIP} proved the following.

\begin{theorem}\emph{\cite [Theorem 1.3]{CIP}} Let $E\subset \mathbb{Z}_q^{d}$, where $q=p^l$. Suppose $$|E|\gg l(l+1)q^{\frac{(2l-1)}{2l}d+ \frac{1}{2l}}.$$ Then the distance set $\Delta(E)$ determined by the points of $E$ satisfies 

 $$ \Delta(E) \supset \mathbb{Z}_q^{*}\cup \{0\}.$$

\end{theorem}

Later in \cite{HV}, Hieu and Vinh proved the following distance result  in the context of finite cyclic rings.
\begin{theorem}\emph{\cite[Theorem 2.7]{HV}} Let $A\subset \mathbb{Z}_q$ be of cardinality $|A|\gg q^{1-\frac{1}{2r}}$. Then, the size of  the distance set determined by $A^n$ satisfies

$$|\Delta_{\mathbb{Z}_q}(A^n)|\gg \min\left\{ q, \frac{|A|^{2n-1}}{ (rq^{2-\frac{1}{r}})^{n-1}}   \right\}.$$

\end{theorem}

Now, let $R$ denote a finite valuation ring. In this paper, we study a variant of distance problem, namely pinned distance problem,  for Cartesian product subsets $A\times A$ of $R^2$. 

Note that, the method we use to prove the main result of this paper is analogous to the one given by Petridis in \cite[Theorem 1.1]{P}. More precisely, we first see pinned distances at a fixed point in $R^2$ as a point-plane incidence in $R^3$. Then we employ the point-plane incidence bound for multisets in $R^3$ which is recently given by Van The et al.  in \cite[Theorem 2.3]{VTHV}.  This yields  the lower bound for the size of the specified pinned distance set in $R^2$ in Theorem \ref{pinned}. 

We should mention that the distance result we obtain in  Theorem \ref{pinned} recovers  the pinned distance result in   \cite{P} in the finite filed setting.  Also, in the setting of  modules over finite cyclic rings $\mathbb{Z}_q$, it is an improvement on the distance results given in \cite {CIP, HV} for  the Cartesian product sets $A\times A \subset \mathbb{Z}_q^2.$


\vskip.125in

 Before stating the main theorem,  let us recall some necessary definitions.

\vskip.125in
\subsection {Notation.} 
\vskip.125in

We note that a detailed definition of  finite valuation ring can be found in \cite{N}. In order to make the statements precise and self contained, we will review the definition and provide some key examples in this note. A finite valuation ring is a finite principal ideal domain which is local. Given a finite valuation ring $R$, we associate  two parameters $q$ and $r$ to $R$ as follows. Let the maximal ideal $M$ of $R$ to be of the form $M=(\pi)$, where $\pi$ is the uniformizer of $R$, i.e. a non unit defined up to a unit of $R$.  Let $r$ be the nilpotency degree of $\pi$, that is the smallest positive integer with the property that $\pi^r=0$ and $q$ be the size of the residue field $F=R/(\pi)$. Therefore, $R$ has the filtration
$$R\supset(\pi)\supset (\pi^2)\dots\supset (\pi^r)=0,$$
where $|R|=q^r$.  Some examples of finite valuation rings are as follows.
\begin{enumerate}
\item Finite fields $\mathbb{F}_q$, where $q=p^k$ is a prime power. 
\item Finite cyclic rings $\mathbb{Z}_{p^k}$, where $p$ is a prime.
\item Function fields $F[x]/(f^k)$, where $F$ is a finite field and $f$ is an irreducible polynomial in $F[x]$. 
\item $\mathcal{O}/(p^k)$, where $\mathcal{O}$ is the ring of integers in a number field and $p$ is a prime in $\mathcal{O}$.
\end{enumerate}

Let us also write some of the examples above with parameters $q$ and $r$ as stated in the definition. Note that  for the finite field $R=\mathbb{F}_{p^k}$, $p$ is a prime, we have $q=p^k$ and $r=1$. And for the finite cyclic ring $\mathbb{Z}_{p^k}$ we have the filtration
$$\mathbb{Z}_{p^k}\supset (p)\supset (p^2)\supset\dots\supset (p^k)=0.$$
Hence $r=k$ and $q=|\mathbb{Z}_{p^k}/(p)|=p$ in this case.

Next we recall the notion of distance in this context. For two points $u=(u_1,\dots,u_d)$ and $v=(v_1, \dots v_d)$ in $R^d$, the distance between them is given by
$$||u-v||=(u_1-v_1)^2+\dots+(u_d-v_d)^2.$$

For a subset $E\subset R^d$, the distance set determined by $E$ is 
$$\Delta(E)=\{||u-v||: u,v \in E\},$$
and the distance set pinned at a fixed point $u$ of $E$  is defined by
$$\Delta_u(E)=\{||u-v||: v\in E\}.$$

Throughout $R$ will denote a finite valuation of order $q^r$, where $q$ is an odd prime power. $X\gg Y$ means that there exists an absolute constant $c$ such that $X\ge cY$, and $"\ll"$ is defined similarly.

\vskip.125in
\subsection{Statement of Main Result}
\vskip.125in

Our main result is the following theorem.
\vskip.125in
\begin{theorem}\label{pinned}
Let $R$ be a finite valuation ring of order $q^r$, q is an odd prime power, and $A\subset R$. There exists  a point $u\in A\times A\subset R^2$ such that 
$$|\Delta_u(A\times A)|\gg \min\left\{q^r, \frac{|A|^3}{q^{2r-1}}\right\} .$$  In particular, if $|A|\ge q^{r-\frac{1}{3}}$, then $\Delta_u(A\times A)\gg q^r$ for some $u\in A\times A$ and hence $A\times A$ determines a positive proportion of all possible distances.
\end{theorem}

 \vskip.125in
 
 \begin{remark} Let $R=\mathbb{F}_p$, where $p$ is an odd prime.   Note that in this case we can take $q=p$ and $r=1$  in Theorem \ref{pinned}  and conclude that if $A\subset \mathbb{F}_p$, then there exists $u\in A\times A\subset \mathbb{F}_p^2$ such that
$$|\Delta_u(A\times A)|\gg \min\left\{p, \frac{|A|^3}{p}\right\}.$$
In particular, if $|A|>p^{\frac{2}{3}},$ then $|\Delta_u(A\times A)|\gg p$ for some $u\in A\times A$. This result matches with the result of Petridis given in \cite[Theorem 1.1]{P} in the context of prime fields and generalize it to the broader context of finite valuation rings.
\end{remark}

\vskip.125in
\begin{remark}
We note that the result in \cite [Theorem 1.3]{CIP} in the special case $d=2$ implies that if $E\subset \mathbb{Z}_q^2$, where $q=p^l$, and
$|E|\gg l(l+1)q^{2-\frac{1}{2l}}$, then 
$$\Delta(E)\supset \mathbb{Z}_q^{*}.$$
On the other hand, Theorem \ref{pinned} implies that if $E= A\times A\subset \mathbb{Z}_q^2$, where $q=p^l$, and $|E|=|A\times A|\ge q^{2-\frac{2}{3l}}$, then $|\Delta(E)|= |\Delta(A\times A)|\gg q.$

Therefore, in terms of getting a positive proportion of all possible distances,  Theorem \ref{pinned} improves the result in \cite[Theorem 1.3]{CIP} for  Cartesian product  sets of the form $A\times A\subset \mathbb{Z}_q^2$.
 \end{remark}

\vskip.125in

\begin{remark}
In \cite[Theorem 2.7]{HV}, for $n=2$, the following result was obtained for subsets of finite cyclic rings. Let $A\subset \mathbb{Z}_q$ be of cardinality $|A|\gtrsim q^{1-1/2r},$ where $q=p^r$. Then the number of distances determined by $A\times A$ satisfies

$$|\Delta_{\mathbb{Z}_q}(A\times A)| \gtrsim \min \left\{ q, \frac{|A|^{3}}{rq^{2-1/r}} \right\} $$

Theorem \ref{pinned} can be seen as a generalization of this result to finite valuation rings and a slight improvement in the context of finite cyclic rings.
\end{remark}

\vskip.125in

\section{Proof of Theorem \ref{pinned}} 
\vskip.125in

For the proof of Theorem \ref{pinned}, we will need the following lemma from \cite{P}. We note here that  though Petridis stated Lemma \ref{average} for subsets of finite fields $\mathbb{F}_q$,  it can be readily checked that the same proof applies for subsets of any finite valuation ring.

\vskip.125in
\begin{lemma}\label{average} Let $E\subset R^2$ and $N$ be the number of solutions to 
\begin{eqnarray}\label{dot}
2u\cdot(v-w)+||w||-||v||=0,
\end{eqnarray}
where $u, v,w \in E$. Then there exists $u\in E$ such that $|\Delta_u(E)|\ge \frac{|E|^3}{N}.$
\end{lemma}



We will also use the following point-plane incidence bound in $R^3$ from \cite{VTHV}.

\begin{theorem} \emph{\cite[Theorem 2.3]{VTHV}}\label{incidence}
Let $Q, \Pi$ be weighted set of points and planes in $R^3$ with the weighted integer function $w$, both total weight $W$. Suppose that maximum weights are bounded by $w_{0}\ge 1$
Let the number of weighted incidences be 
$$ I_{w}= \sum _{q\in Q,\; \pi\in \Pi} w(q)w(\pi)\delta_{q\pi},$$
where 
\begin{equation*}
 \delta_{q\pi}  = \left\{
    \begin{array}{rl}
      1& \text{if } q\in\pi,\\
      0 & \text{if } q\notin \pi.
    \end{array} \right.
\end{equation*}
 Then the number $I_{w}$ of weighted incidences is bounded as follows:
$$ I_{w}=\sum_{q\in \pi} w(q)w(\pi)\ll \frac{1}{q^r} W^2+q^{2r-1} W.$$ 
 \end{theorem}

\vskip.125in
\begin{proof}[Proof of Theorem \ref{pinned}]
We first note that if we write $u=(u_1, u_2)$, $v=(v_1,v_2)$ and $w=(w_1, w_2)$, where $u_i, v_i, w_i\in A$, then the equation (\ref{dot}) can be written as 
\begin{eqnarray*}
2u_{1}(v_{1}-w_{1})+2u_{2}(v_{2}-w_{2})+(w_{2}^2-v_{2}^2)=v_{1}^2-w_{1}^2
\end{eqnarray*}
which can be restated as 
\begin{eqnarray}\label{dot1}
(2u_1,v_2-w_2, w_2^2-v_2^2)\cdot(v_1-w_1,2u_2,1)=v_1^2-w_1^2.
\end{eqnarray}
Next we define a set of points $Q$ and a set of planes $\Pi$ in $R^3$ as follows:
$$Q=\{(2u_1,v_2-w_2,w_2^2-v_2^2):u_1,v_2,w_2\in A\}$$
and
$$\Pi=\{\{x\in R^3: x\cdot (v_1-w_1,2u_2,1)=v_1^2-w_1^2\}: v_1, w_1, u_2\in A\}$$

Then it follows that the number of incidences $\left|I(Q,\Pi)\right|$ between $Q$ and $\Pi$  is equal to the number of solutions of the equation (\ref{dot1}) which is $N$ in Lemma \ref{average}. 

Note that  the total weight  of $Q$ and $\Pi$ are both  $W=|A|^3$. Hence, Theorem \ref{incidence} implies that
\begin{eqnarray*}
N&=&\left|I(Q,\Pi)\right|\\
&\le& \frac{1}{q^r}W^2+ q^{2r-1}W\\
&=& \frac{1}{q^r}|A|^6+q^{2r-1}|A|^3.
\end{eqnarray*}
Therefore, by Lemma \ref{average}, there exists $u\in A\times A$ such that 
\begin{eqnarray*}
|\Delta_u(A\times A)|\ge \frac{|A|^6}{N}\gg\min\left\{q^r, \frac{|A|^3}{q^{2r-1}}\right\} 
\end{eqnarray*}
which completes the proof of Theorem \ref{pinned}.
\end{proof}

\vskip.125in
\noindent \textbf{Acknowledgments.} The author  would like to thank Brendan Murphy for valuable comments.
\vskip.125in


 \end{document}